\documentclass[11pt,a4paper,reqno]{article}
\usepackage[latin1]{inputenc}  
\usepackage[T1]{fontenc}
\usepackage[english]{babel}
\usepackage{graphicx}
\usepackage{amsfonts,amssymb,latexsym,a4wide,xspace}
\usepackage{amsmath,delarray,enumerate,calc,amsthm}
\usepackage{bbm}
\usepackage{txfonts}
\usepackage{paralist}
\usepackage[mathscr]{euscript}
\usepackage{xcolor}
\usepackage{authblk}
\usepackage{tikz-cd}
\usetikzlibrary{matrix}

\newcommand{\id}{\operatorname{id}}
\newcommand{\supp}{\operatorname{supp}}

\newcommand{\cB}{{\mathcal B}}

\newcommand{\cF}{{\mathcal F}}
\newcommand{\cG}{{\mathcal G}}

\newcommand{\cO}{{\mathcal O}}

\newcommand{\cK}{{\mathcal K}}
\newcommand{\cL}{{\mathcal L}}
\newcommand{\cM}{{\mathcal M}}
\newcommand{\cN}{{\mathcal N}}
\newcommand{\cS}{{\mathcal S}}
\newcommand{\cU}{{\mathcal U}}
\newcommand{\cX}{{\mathcal X}}
\newcommand{\LL}{{\mathscr L}}
\newcommand{\BB}{{\mathscr B}}
\newcommand{\R}{{\mathbbm R}}
\newcommand{\N}{{\mathbbm N}}

\newcommand{\Z}{{\mathbbm Z}}

\newcommand{\Seq}{S^{eq}}
\newcommand{\inn}{\operatorname{int}}

\newcommand{\card}{\operatorname{card}}

\newcommand{\hx}{{\hat{x}}}
\newcommand{\hX}{{\hat{X}}}

\newcommand{\Xnull}{X_0}
\newcommand{\cOsq}{\cO^{\,2}}
\newcommand{\cOnull}{\cO_0}
\newcommand{\cOnullsq}{\cO_0^{\,2}}
\newcommand{\Deltanull}{\Delta_0}
\newcommand{\Qnull}{Q_0}

\newcommand{\nuW}{\nu_{\scriptscriptstyle W}}
\newcommand{\nuWprime}{\nu_{\scriptscriptstyle W'}}

\newcommand{\nuG}{\nu^{\scriptscriptstyle G}}
\newcommand{\nuWG}{\nu_{\scriptscriptstyle W}^{\scriptscriptstyle G}}

\newcommand{\MW}{\cM_{\scriptscriptstyle W}}
\newcommand{\MG}{\cM^{\scriptscriptstyle G}}

\newcommand{\MWG}{\cM^{\scriptscriptstyle G}_{\scriptscriptstyle W}}

\newcommand{\GMWG}{\cG\cM^{\scriptscriptstyle G}_{\scriptscriptstyle W}}

\newcommand{\piG}{\pi^{\scriptscriptstyle G}}
\newcommand{\tpiG}{\tilde\pi^{\scriptscriptstyle G}}
\newcommand{\piH}{\pi^{\scriptscriptstyle H}}

\newcommand{\piGH}{\pi^{\scriptscriptstyle G\times H}}
\newcommand{\pihX}{\pi^{\scriptscriptstyle \hX}}
\newcommand{\pihXG}{\pi^{\scriptscriptstyle \hX\times G}}
\newcommand{\hpiG}{\hat\pi^{\scriptscriptstyle G}}

\newcommand{\0}{\underline{0}}
\newcommand{\QM}{Q_{\cM}}

\newcommand{\QMG}{Q_{\cM^G}}

\newcommand{\oneone}{1-1\xspace}
\newcommand{\dL}{\mathrm{dens}(\LL)}
\newcommand{\cMG}{\cM^{\scriptscriptstyle G}}
\newcommand{\XG}{X^{\scriptscriptstyle G}}
\newcommand{\SH}{\cS_{\scriptscriptstyle H}}
\newcommand{\ZG}{Z^{\scriptscriptstyle G}}
\newcommand{\piZ}{\pi^{\scriptscriptstyle Z}}
\newcommand{\piZG}{\pi^{\scriptscriptstyle \ZG}}
\newcommand{\gen}{\overset{\scriptscriptstyle gen}{\to}}

\newcommand{\KW}{\cK_{\scriptscriptstyle W}}

\newcommand{\ddelta}{\boldsymbol{\delta}}

\newtheorem {definition}{Definition}[section] 
\newtheorem {lemma}{Lemma}[section]
\newtheorem{theorem}{Theorem}[section]
\newtheorem {bemerkung}{Remark}[section]
\newtheorem{proposition}{Proposition}[section]
\newtheorem {corollary}{Corollary}[section]
\newtheorem{beispiel}{Example}[section]
\newenvironment{remark} {\begin{bemerkung} \normalfont }{\end{bemerkung}}
\newenvironment{example} {\begin{beispiel} \normalfont }{\end{beispiel}}

\begin{document}

\title{Maximal equicontinuous generic factors, weak model sets, and centralizers of some $\cB$-free subshifts
\thanks{This publication combines a modification of the preprint arxiv:1610.03998 with some new material in Section~\ref{sec:B-free}.}}
\author{Gerhard Keller
\thanks{I am indebted to A.~Dymek and M.~Lema\'n{}czyk for helpful comments on this manuscript.}}
\affil{Department Mathematik, Universit\"at Erlangen-N\"urnberg}
\date{\today}


\maketitle

\begin{abstract}
The orbit closures of regular model sets generated from a cut-and-project scheme given by a co-compact lattice $\cL\subset G\times H$ and  compact  and aperiodic window $W\subseteq H$, have the maximal equicontinuous factor (MEF) $(G\times H)/\cL$, if the window is toplogically regular. This picture breaks down completely, when the window has empty interior, in which case the MEF is always trivial, although $(G\times H)/\cL$ continues to be the Kronecker factor for the Mirsky measure. As this situation occurs for many interesting examples like the square-free numbers or the visible lattice points, there is some need for a slightly weaker concept of topological factors that is still strong enough to capture basic properties of the system. Here we propose to use the concept of a generic factor \cite{HuangYe2012} for this purpose. For so called ergodic topological dynamical systems we prove the existence of a maximal equicontinuous generic factor (MEGF) and characterize it in terms of the regional proximal relation. For such systems we also show that the MEGF is trivial if and only if the system is topologically weakly mixing. This part of the paper profits strongly from previous work by McMahon \cite{McMahon1978} and Auslander \cite{Auslander1988}.  
In Section~3 we show that $(G\times H)/\cL$ is indeed the MEGF of the orbit closure of each weak model set with an aperiodic Haar regular window, and in Section~\ref{sec:B-free} we apply this fact to give an alternative proof of the finding \cite{Mentzen2017} that the centralizer of any $\cB$-free dynamical system of Erd\"os type is trivial.
\end{abstract}

\renewcommand{\thefootnote}{\fnsymbol{footnote}} 
\renewcommand{\thefootnote}{\arabic{footnote}}


\section{Introduction}

Let $G$ and $H$ be locally compact second countable groups. In many examples, $G=\Z^d$ or $\R^d$, whereas $H$ will often be a more general group. Each pair $(\LL,W)$, consisting of a cocompact lattice $\LL\subset G\times H$ and a  compact subset $W$ of $H$, also called the window, defines a weak model set $\Lambda(\LL,W)$ as the set of all points $x_G\in G$, for which there exists a point $x_H\in W$ such that $(x_G,x_H)\in\LL$. There is an abundant literature on model sets,
see e.g. the collection of references cited in \cite{BaakeGrimm13}.
Many of these sets exemplify \emph{aperiodic order}, a concept which, so far, is mostly defined by a wealth of examples \cite{Baake2002,Baake2016}.
The following seems to be a common feature of all of them:
No $g\in G\setminus\{0\}$ satisfies $g+\Lambda(\LL,W)=\Lambda(\LL,W)$, but
the orbit closure $\overline{\{g+\Lambda(\LL,W):g\in G\}}$ as a $G$-dynamical system has a nontrivial \emph{maximal equicontinuous factor (MEF)} and/or a nontrivial \emph{Kronecker factor (KF)} capturing the quasiperiodic aspects of the dynamics.

Many of the simpler examples are uniquely ergodic, so that one can talk unambigously about their KF, and quite often this KF is just the MEF equipped with its Haar measure. But, more recently, dynamically richer examples, like  the set of square-free numbers \cite{CS13,Peckner2012}, the set of visible lattice points \cite{BaakeHuck14} and their generalizations \cite{BKKL2015}, have attracted much attention. They share the common feature that the orbit closure
$\overline{\{g+\Lambda(\LL,W):g\in G\}}$ has a fixed point, so that the MEF must be trivial, whereas there are plenty of invariant measures that have non-trivial KFs. Equipped with a very natural invariant measure (called Mirsky measure in some cases) these systems are actually isomorphic to their KFs. The approach to the dynamics of weak model sets from \cite{KR2015} encompasses all these examples. It suggests that the notion of \emph{maximal equicontinuous generic factor (MEGF)} might be an appropriate unifying concept. Equicontinuous generic factors for abstract topological dynamical systems were studied in \cite{HuangYe2012}, see also \cite{Downarowicz1998} for generic eigenfunctions.

In the next section we review some facts on equicontinuous generic factors and prove the existence of a unique MEGF for \emph{ergodic topological dynamical systems}, i.e. for continuous actions of infinite (abstract) groups on compact metric spaces, which admit an ergodic invariant probability measure with full topological support (Theorem~\ref{theo:max-factor-prelim}). The proof is inspired by work of McMahon \cite{McMahon1978} and Auslander \cite{Auslander1988}. If the acting group of such a system is abelian, we also show that the MEGF is trivial if and only if the system is topologically weakly mixing.

In Section~\ref{sec:MEF-for-WMS} we test the potential usefulness of this construction
on the dynamics on weak model sets, where there are good reasons to expect that this construction leads to the canonical $G$-action on the group $(G\times H)/\LL$, see e.g. \cite{KR2015,BKKL2015}. We will see that $\left((G\times H)/ \LL,G\right)$ is indeed both, the MEGF and the KF of the orbit closure 
of any weak model set that is generic for its Mirsky measure (Theorem~\ref{theo:cut-project-MEGF}). 
It should be noted, however, that in general ergodic topological dynamical systems the MEGF and the KF need not coincide (Example~\ref{example:B-free-KF-MEGF}).

In Section~\ref{sec:B-free} we restrict to the dynamics of very particular weak model sets, namely to $\cB$-free dynamics, see e.g. \cite{Kulaga-Przymus2014,BKKL2015,KKL2017}. If the set $\cB\subseteq\N$ defining such a system is taut and contains an infinite pairwise co-prime subset (this is the case e.g. for the square free numbers), then centralizer of the system is trivial (Theorem~\ref{theo:trivial-centralizer}).

\section{Equicontinuous generic factors}\label{sec:EGF}

Let $(X,G)$ be a topological dynamical system with an abstract group $G$ acting by homeomorphisms on a compact metrizable space  $X$. For this action we adopt the short cut notation $x\mapsto gx$. By $Gx$ we denote $\{gx:g\in G\}$.

\begin{definition}\label{def:ergodic}\cite{Glasner1993,Glasner2003}
The system $(X,G)$ is \emph{ergodic}, if 
there exists an ergodic $G$-invariant Borel probability measure $\lambda$ on $X$ with $\supp(\lambda)=X$.
\end{definition}

\begin{remark}\label{remark:first}
\begin{compactenum}[a)]
\item Observe that the following properties are equivalent:
\begin{compactenum}[(i)]
\item $(X,G)$ is ergodic.
\item There exists a $G$-invariant Borel probability measure $\lambda$ on $X$ with $\supp(\lambda)=X$ and $\lambda(X_t)=1$.
\end{compactenum}
The implication (i) $\Rightarrow$ (ii) follows immediately from 
\cite[Lemma 2.4]{Keynes1969}. For the reverse implication note that
\cite[Lemma 2.4]{Keynes1969} guarantees that $\lambda$ is ``closed ergodic'' so that
the ergodicity of $(X,G)$ follows from
\cite[Proposition 2.6]{Keynes1969}.
\item In particular, each ergodic $(X,G)$ is topologically transitive.
\item
The mere existence of a (non-ergodic) $G$-invariant probability measure with full support (i.e. being an $E$-system in the terminology of \cite{Glasner1993,Glasner2003}) is a weaker property than ergodicty of $(X,G)$
\cite{Weiss1971}.
\end{compactenum}
\end{remark}

From now on we assume that 
$(X,G)$ is \emph{topologically transitive}. 
(Ergodicity will be required only when it is needed.)
Then the set
\begin{displaymath}
X_t:=\{x\in X:\overline{Gx}=X\}
\end{displaymath}
of transitive points is $G$-invariant and residual, see e.g. \cite[Proposition 1]{Cairns2007}. 
\begin{remark}\label{remark:second-countable}
As a topological space (endowed with the induced topology from the compact metric space $X$), $X_t$ is second countable and in particular separable \cite[Theorem 30.3]{Munkres-book}. The same is true for each 
$G$-invariant subset $\Xnull\subseteq X_t$ with the induced topology.
\end{remark}

For topological dynamical systems with $G=\Z$ (even with $G=\N$), Huang and Ye \cite{HuangYe2012} introduced the notion of an equicontinuous generic factor, that we adapt here to general $G$:

\begin{definition}\label{def:MEGF}
\begin{compactenum}[a)]
\item The system $(Y,G)$ is an \emph{equicontinuous generic factor} of $(X,G)$, if 
$(Y,G)$ is an  equicontinuous, transitive (and hence minimal) system, and if
there is a continuous map $\pi:X_t\to Y$ equivariant under the action of $G$. (As $(Y,G)$ is minimal, the image $\pi(X_t)$ is dense in $Y$.) We also write $\pi:(X,G)\gen(Y,G)$.
\item An \emph{equicontinuous generic factor} $\pi:(X,G)\gen(Z,G)$ is \emph{maximal}, if the following holds:\\
If $\pi_Y:(X,G)\gen(Y,G)$ is another equicontinuous generic factor, then there is a factor map $\pi':(Z,G)\to (Y,G)$ such that $\pi_Y=\pi'\circ\pi$.
\end{compactenum}
\end{definition}

The following theorem is the main result of this paper:

\begin{theorem}\label{theo:max-factor-prelim}
An ergodic system $(X,G)$ has a unique maximal equicontinuous generic factor (MEGF). 
\end{theorem}
\noindent
A generalization of this theorem, which also provides more details of the construction of the factor map, is stated and proved below in Subsection~\ref{subsec:existence-MEGF}.

\begin{remark}
An equicontinuous generic factor map $\pi$ is in particular a continuous map from a dense subset $X_t$ of $X$ to $Y$. If $X$ and $Y$ are both compact, such a map can always be extended to a measurable map from $X$ to $Y$, continuous at each point of $X_t$. Indeed,  
denote by $\overline \Pi:=\overline{\{(x,\pi(x)): x\in X_t\}}$ the closure of the graph of $\pi$ in $X\times Y$. The multivalued map $\phi:x\mapsto \overline \Pi_x$ that associates to each point $x\in X$  the (compact) $x$-section of $\overline \Pi$ is upper semi-continuous and hence measurable \cite[Corollary III.3]{CV1977} so that there is a measurable selector $\tilde{\pi}:X\to Y$ such that $\tilde{\pi}(x)\in \overline \Pi_x$ for each $x\in X$ \cite[Theorem III.6]{CV1977}. As $\overline \Pi_x=\{\pi(x)\}$ at all $x\in X_t$ ($\pi$ is continuous at all these points!), $\tilde{\pi}$ extends $\pi$, and as the graph of $\tilde{\pi}$ is contained in $\overline \Pi$, $\tilde{\pi}$ is continuous at all $x\in X_t$. 

Note, however, that in general $\tilde{\pi}$ cannot be chosen to be equivariant under the actions of $G$, because fixed points are always mapped to fixed points under equivariant maps.
\end{remark}

For the case $G=\Z$ (and also $G=\N$) and without assuming that the system $(X,G)$ is ergodic, Huang and Ye \cite[Theorem 3.8]{HuangYe2012} proved that a system $(X,G)$ has a trivial MEGF if and only if it is weakly scattering, and that if $(X,G)$ is an $E$-system, then this happens if and only if the system is weakly mixing. Here we prove,
for general acting groups $G$ and ergodic $(X,G)$, that such a system is weakly mixing if and only if it has a trivial MEGF.

\begin{theorem}\label {theo:weak-mixing}
Assume that
$(X,G)$ is topologically transitive.
\begin{compactenum}[a)]
\item
If $(X,G)$ is weakly scattering, i.e. if $(X\times Z,G)$ (the product action) is topologically transitive for each minimal equicontinuous $(Z,G)$, then the maximal equicontinuous generic factor of $(X,G)$ is trivial (i.e. a singleton).
\item If $(X,G)$ is ergodic and has a trivial maximal equicontinuous generic factor, then $(X,G)$ is topologically weakly mixing.
\end{compactenum}
\end{theorem}
\noindent For the proof see Subsection~\ref{subsec:weak-mixing}.
As weak mixing implies weak scattering, we have the following corollary:
\begin{corollary}
Assume that
$(X,G)$ is an ergodic topologically transitive system. Then $(X,G)$ is
weakly scattering if and only if it is topologically weakly mixing if and only if its maximal equicontinuous generic factor is trivial.
\end{corollary}

\subsection{The regional proximal and the equicontinuous structure relation}

\paragraph{The equicontinuous structure relation}
During all of this note, $\Xnull$ denotes a dense, $G$-invariant subset of $X_t$.
We endow $X_t$ and $\Xnull$ with the induced topology inherited from $X$, and we denote by
$\cO,\cO_t$ and $\cOnull$ these topologies on $X,X_t$ and $\Xnull$, respectively.
The corresponding product topologies on $X^2,X_t^2$ and $\Xnull^2$ are denoted by
$\cOsq,\cOsq_t$ and $\cOnullsq$, respectively. We recall the following elementary observation:

\begin{remark}\label{rem:elementary-1}
For $A\subseteq X$ and $B\subseteq X^2$ we have
\footnote{Here is the proof for $A\subseteq X$. The one for $B\subseteq X^2$ is nearly identical:
\begin{equation*}
\begin{split}
x\in \overline{A\cap\Xnull}^{\cOnull}
&\Leftrightarrow
x\in \Xnull\ \wedge\ \forall V\in \cO:\ x\in V\cap \Xnull\Rightarrow (A\cap\Xnull)\cap(V \cap\Xnull)\neq\emptyset\\
&\Leftrightarrow
x\in \Xnull\ \wedge\ \forall V\in \cO:\ x\in V\Rightarrow (A \cap\Xnull)\cap V\neq\emptyset\\
&\Leftrightarrow
x\in\Xnull \ \wedge\ x\in\overline{A\cap \Xnull}^\cO.
\end{split}
\end{equation*}
}
\begin{equation}\label{eq:new-1}
\overline{A\cap\Xnull}^{\cOnull}=\overline{A\cap \Xnull}^\cO\cap\Xnull
\quad\text{and}\quad
\overline{B\cap\Xnull^2}^{\cOnullsq}=\overline{B\cap \Xnull^2}^{\cOsq}\cap\Xnull^2\,.
\end{equation}
For $U\in\cO$ and $V\in\cOsq$ this implies
\begin{equation}\label{eq:new-2}
\overline{U\cap\Xnull}^{\cOnull}=\overline{U}^\cO\cap\Xnull
\quad\text{and}\quad
\overline{V\cap\Xnull^2}^{\cOnullsq}=\overline{V}^{\cOsq}\cap\Xnull^2\,.
\end{equation}
If $X_0=X_t$, eqn.~\eqref{eq:new-1} yields $\overline{A\cap X_t}^{\cO_t}=\overline{A\cap X_t}^\cO\cap X_t$,
and this identity applied to $A=X_0$ gives
\begin{equation*}
\overline{\Xnull}^{\cO_t}=\overline{\Xnull\cap X_t}^{\cO_t}
=\overline{X_0\cap X_t}^\cO\cap X_t
=\overline{\Xnull}^\cO\cap X_t.
\end{equation*}
\end{remark}

For a homomorphism (i.e. a continuous G-equivariant map) $\pi:\Xnull\to Y$, where $Y$ is any compact metric group on which $G$ acts minimally by translation (so that $\pi(\Xnull)$ is dense in $Y$ \footnote{Observe that, vice versa, denseness of $\pi(\Xnull)$ in $Y$ implies minimality of the action of $G$ on $Y$.}), denote
\begin{equation}\label{eq:Stpi-def}
S^\pi_{0}:=\{(x,x')\in \Xnull^2: \pi(x)=\pi(x')\}\ ,
\end{equation}
and let
\begin{equation}\label{eq:Steq-def}
\Seq_{0}:=\bigcap_{\pi:\Xnull\to Y}S^\pi_{0}\ .
\end{equation}
(The intersection is taken over all such spaces $Y$ and all homomorphisms $\pi:\Xnull\to Y$.)
We call $\Seq_{0}$ the equicontinuous structure relation on $\Xnull$. It is obviously an $\cOnull$-closed, $G$-invariant equivalence relation on $\Xnull$.
\begin{remark}
Recall that each topologically transitive, equicontinuous  $G$-action on a compact metric space $Y$ induces the structure of an abelian group on $Y$, such that $Y$ (with its original topology) becomes a topological group, on which $G$ acts by translation. \footnote{This fact goes back to work of Ellis. A comprehensive presentation of the corresponding circle of ideas can be found in \cite[Sec.~3.1]{ABKL14}.} In particular, using an equivalent metric on $Y$, if necessary, we can always assume that the action of $G$ on $Y$ is isometric. \footnote{Denote by $d$ the given metric on $Y$. Then $d'(y_1,y_2):=\sup\{d(gy_1,gy_2): g\in G\}$ has the desired properties.}
\end{remark}

\paragraph{The regional proximal relation}
Denote
\begin{displaymath}
\Deltanull:=\{(x,x): x\in \Xnull\}\ ,
\end{displaymath}
\begin{displaymath}
\cU_{0}:=\left\{U_{0}:=U\cap \Xnull^2: U\in\cOsq\text{ $G$-invariant, }\Deltanull\subset U\right\}\ ,
\end{displaymath}
let
\begin{displaymath}
\Qnull:=\bigcap_{U_{0}\in\cU_{0}}\overline{U_{0}}\cap \Xnull^2\ ,
\end{displaymath}
and denote by $S^\ast_{0}$ the smallest $\cOnullsq$-closed, $G$-invariant equivalence relation on $\Xnull^2$ containing $\Qnull$.
Because of \eqref{eq:new-2} in Remark~\ref{rem:elementary-1}, it does not matter how the topological hull operation on $U_{0}\subseteq\Xnull^2$ is interpreted.

\begin{lemma}\label{lemma:PR-characterization}
$\Qnull$ is the regional proximal relation of the non-compact dynamical system $(\Xnull,G)$, i.e.
\begin{equation}\label{eq:PR-characterization}
\begin{split}
\Qnull=&\left\{(x,y)\in \Xnull^2: \text{for all neighbourhoods $A\in\cO$ of $x$, $B\in\cO$ of $y$ and $V\in\cOsq$ of $\Deltanull$ there exist}\right.
\\
&\left.\text{\hspace*{3cm}$x'\in A\cap \Xnull$, $y'\in B\cap \Xnull$ and $g\in G$  s.t. $(gx',gy')\in V$}\right\}.
\end{split}
\end{equation}
\end{lemma}
\begin{proof}
Suppose first that $(x,y)\in \Qnull\subseteq\Xnull^2$ and that $A,B$ and $V$ are neighbourhoods
as in (\ref{eq:PR-characterization}). Define $U:=\bigcup_{g\in G}gV$. Then $U\in\cOsq$ is a $G$-invariant neighbourhood of $\Deltanull$, i.e. $U_{0}\in\cU_{0}$. Hence $(x,y)\in\overline{U_{0}}\cap\Xnull^2$, and there exists $(x',y')\in U_{0}\cap(A\times B)$. So there is $g\in G$ such that 
$(x',y')\in gV\cap(A\times B)\cap(\Xnull\times \Xnull)$, which means that $x'\in A\cap \Xnull$, $y'\in B\cap \Xnull$ and $(g^{-1}x',g^{-1}y')\in V$.

Conversely, suppose that $(x,y)$ belongs to the set on the r.h.s.~of (\ref{eq:PR-characterization}), and consider any $U_{0}=U\cap\Xnull^2\in\cU_{0}$ with a $G$-invariant neighbourhood $U\in\cOsq$  of $\Deltanull$. Given any neighbourhood $O\in\cOsq$ of $(x,y)$, fix neighbourhoods $A\in\cO$ of $x$ and $B\in\cO$ of $y$ such that $A\times B\subseteq O$. Then there are $x'\in A\cap \Xnull$, $y'\in B\times \Xnull$ and $g\in G$ such that $(gx',gy')\in U$, i.e. $(x',y')\in g^{-1}U\cap (\Xnull\times \Xnull)=U_{0}$. Hence $(x,y)\in\overline{U_{0}}\cap \Xnull^2$, and as this holds for each $U_{0}\in\cU_{0}$, we have $(x,y)\in \Qnull$.
\end{proof}

\begin{remark}
The first part of the proof goes through without any changes, if $V\in\cO^2$ just has non-empty intersection with $\Delta_0$, because also in this case the set $U\in\cO^2$ is a $G$-invariant neighbourhood of $\Delta_0$: for $(z,z)\in\Delta_0=\Delta\cap(X_0\times X_0)$ there is $\tilde g\in G$ such that $(\tilde{g}z,\tilde{g}z)\in V$, so that $(z,z)\in U$.
Therefore this lemma shows that (\ref{eq:PR-characterization}) is  the relation $Q_m(\varphi)$ of McMahon \cite{McMahon1978}, when the setting described on \cite[p.227]{McMahon1978} is specialized to the situation treated here, which is the special case where McMahon's $Z$ is the trivial one-point system and where his $X$ and $Y$ coincide. 
In that case his notion of a section collapses to that of a $G$-invariant Borel probability measure on $X$, his set $R_m(\varphi)$ coincides with our $X_{t}\times X_{t}$, and
the definition of his $Q_m(\varphi)$ coincides with that of our $Q_{t}$ and his
$S_m(\varphi)$ is our $S_t^*$.

Note also that his set $X_m$ is our $X_t$, and that $X_t$ is a Borel set under our assumptions. This setting, for the special case of minimal dynamics, is reproduced in Auslander's book \cite{Auslander1988}.
\end{remark}

\paragraph{Inclusions between the various relations}
\begin{lemma}\label{lemma:2.2}
$\Qnull\subseteq \Seq_{0}$ and hence also $S^\ast_{0}\subseteq\Seq_{0}$.
\end{lemma}
\begin{proof}
Let $\pi:\Xnull\to Y$ be as in \eqref{eq:Steq-def},
and recall that w.l.o.g. we can assume that the action of $G$ on $Y$ is isometric.
Let $(x,y)\in \Qnull$. Suppose for a contradiction that $\pi(x)\neq\pi(y)$. 
Let $\delta:=d(\pi(x),\pi(y))>0$.
There are neighbourhoods $A\in\cO$ of $x$ and $B\in\cO$ of $y$ such that 
\begin{displaymath}
d(\pi(gx'),\pi(gy'))
=
d(g\pi(x'),g\pi(y'))
=
d(\pi(x'),\pi(y'))
>\delta/2
\end{displaymath}
for all $x'\in A\cap \Xnull$, $y'\in B\cap \Xnull$ and $g\in G$.

Let $M:=\{(y,y'): y,y'\in Y,\ d(y,y')<\delta/2\}$. The set $M$ is an open neighbourhood of the diagonal in $Y\times Y$, and as the metric is translation invariant, the set $M$ is $G$-invariant. Furthermore,
$(\pi(gx'),\pi(gy'))\not\in {M}$ for all
$x'\in A\cap \Xnull$, $y'\in B\cap \Xnull$ and $g\in G$. Let  $\tilde V:=\bigcup_{g\in G}g((\pi\times\pi)^{-1}M)$.
Then $\tilde{V}\subseteq \Xnull^2$ and $(A\times B)\cap\tilde{V}=\emptyset$.

As $\pi:\Xnull\to Y$ is continuous, $\tilde V\in\cOnullsq$ is a $G$-invariant neighbourhood of $\Deltanull$.
Hence $\tilde V=V\cap\Xnull^2$ for some $V\in\cOsq$. Let $U:=\bigcup_{g\in G}gV$. This set is clearly $\cOsq$-open and $G$-invariant, and it contains $\Deltanull$. Note also that
\begin{displaymath}
U_{0}
=
\left(\bigcup_{g\in G}gV\right)\cap\Xnull^2
=
\bigcup_{g\in G}g\left(V\cap\Xnull^2\right)
=
\bigcup_{g\in G}g\tilde{V}
=\tilde{V}\ .
\end{displaymath}
Hence $U_{0}\in\cU_{0}$ and $(A\times B)\cap U_{0}=(A\times B)\cap \tilde{V}=\emptyset$.
Therefore $(x,y)\not\in\overline{U_{0}}$, which contradicts $(x,y)\in \Qnull$.
\end{proof}

\subsection{The role of invariant measures supported by the transitive points}

We follow McMahon \cite{McMahon1978} and Auslander \cite{Auslander1988} in order to study 
the relation between $\Qnull$ and $S^\ast_{0}$. Although some parts of their proofs carry over directly, we prefer to give full details here. 
\\[3mm]
\textbf{General assumptions and notations}
\begin{compactenum}[$\bullet$]
\item $\cN$ denotes the family of all closed $G$-invariant subsets of $X^2$.
\item For any $N\in\cN$ and $x\in X$ denote by $N_x:=\{y\in X: (x,y)\in N\}$ the $x$-section of $N$.
\item We fix a $G$-invariant Borel probability measure $\lambda$ on $X$. As $X$ is compact metrizable, $\lambda$ is regular.
\end{compactenum}

\begin{lemma}[\cite{McMahon1978,Auslander1988}]\label{lemma:McM-A}
Let $N\in\cN$. Then
\begin{compactenum}[a)]
\item $N_{gx}=gN_x$ and $\lambda(N_{gx})=\lambda(N_x)$ for all $x\in X$ and $g\in G$.
\item The map $x\mapsto\lambda(N_x)$ ($x\in X$) is upper semicontinuous.
\item  $\lambda(N_x)\leqslant\lambda(N_{x'})$ for all $x\in X$ and $x'\in \overline{Gx}$.
\item  $\lambda(N_x)\leqslant\lambda(N_{x'})$ for all $x\in X_t$ and $x'\in X$.
\item  $\lambda(N_x)=\lambda(N_{x'})$ for all $x,x'\in X_t$.
\item The map $x\mapsto \lambda(N_x)$ $(x\in X)$ is continuous at each $x\in X_t$.
\item $\lambda(N_{gx}\triangle N_{gx'})=\lambda(N_x\triangle N_{x'})$ for all $x,x'\in X$ and $g\in G$.
\item The map $(x,x')\mapsto\lambda(N_x\triangle N_{x'})$ ($x,x'\in X$) is continuous
at each $(x,y)\in X_t\times X_t$.
\end{compactenum}
\end{lemma}
\begin{proof}
\begin{compactenum}[a)]
\item $N_{gx}=\{y\in X: (gx,y)\in N\}=\{y\in X: (x,g^{-1}y)\in N\}=\{gy'\in X: (x,y')\in N\}=gN_x$.
\item Let $\epsilon>0$. There is an open neighbourhood of $V\supseteq N_x$ such that $\lambda(V)<\lambda(N_x)+\epsilon$. By compactness of $N$, there is a neighbourhood $U\subseteq X$ of $x$ such that $N_{x'}\subseteq V$ for all $x'\in U$. Hence $\lambda(N_{x'})\leqslant\lambda(V)<\lambda(N_x)+\epsilon$ for all $x'\in U$.
\item This follows from a) and b).
\item This is a special case of c).
\item This follows from d).
\item Let $\lambda_{min}:=\inf_{x'\in X}\lambda(N_{x'})$, let $x\in X_t$ and $\epsilon>0$. In view of b) and d), there is a neighbourhood $U$ of $x$ such that  $\lambda(N_x)=\lambda_{min}\leqslant\lambda(N_{x'})<\lambda(N_x)+ \epsilon$ for each $x'\in U$.
\item This follows from a): $\lambda(N_{gx}\triangle N_{gx'})=\lambda(gN_x\triangle gN_{x'})=\lambda(g(N_x\triangle N_{x'}))=\lambda(N_x\triangle N_{x'})$
\item This follows from f).
\end{compactenum}
\end{proof}

Following \cite{Auslander1988} we define  pseudometrics $d_N$ on $X_t$: \footnote{In \cite{Auslander1988} this is written down for minimal $(X,G)$.} given $N\in\cN$ let
\begin{displaymath}
d_N(x,x'):=\lambda(N_x\triangle N_{x'})\,.
\end{displaymath}
Their restrictions to $\Xnull^2$ yield pseudo-metrics on $\Xnull$.
For $\Xnull\subseteq X_t$ let \footnote{See also \cite[Lemma~1.2]{McMahon1978}.}
\begin{displaymath}
K_{0}(N)=\left\{(x,x')\in \Xnull^2: d_N(x,x')=0\right\}.
\end{displaymath}
If $\Xnull=X_t$, we denote the this set by $K_t(N)$. Observe that
\begin{equation*}
K_{0}(N)=K_t(N)\cap\Xnull^2\,.
\end{equation*}

\begin{remark}\label{remark:discussion}
By Lemma~\ref{lemma:McM-A}, $d_N$ is $G$-invariant and continuous,
so that $K_{0}(N)$ is a $G$-invariant  $\cOnull$-closed equivalence relation on  $\Xnull$.  
Let $Z_N^*:=\Xnull/K_0$ and define 
$d_N^*([x],[y]):=d_N(x,y)$ for $x,y\in X_0$. Then $(Z_N^*,d_N^*)$ is a metric space, and the canonical projection $\pi_{N}:\Xnull\to Z_N^*$ is continuous. As $K_0(N)$ is $G$-invariant, $G$ acts in a canonical way on $Z_N^*$, and this action is isometric. Hence it extends isometrically to the completion of $Z_N^*$, which we denote by $X_{N}$. As $Z_N^*$ is the continuous image of a separable space, it is separable, and so is its completion $(X_{N},d_{N})$. Finally, as $Z_N^*$ is the continuous image of a subset $X_0$ of the set of transitive points, 
also the action of $G$ on $Z_N^*$ is topologically transitive,
 and as that action is equicontinuous, the action of $G$ on $X_{N}$ is in fact minimal.
In order to conclude that $\Seq_{0}\subseteq S_{0}^{\pi_{N}}=K_{0}(N)$, we would need to know that $(X_{N},d_{N})$ is compact. As this space is complete by construction, all that remains to be proved is that it is totally bounded, which is achieved in the following lemma.
\end{remark}

\begin{lemma}\label{lemma:elementary-2}
Let $Z$ be a separable metric space on which $G$ acts isometrically and transitively. If there exists a
finite, non-trivial, $G$-invariant Borel measure $\mu$ on $Z$, then $Z$ is totally bounded.
\end{lemma}
\begin{proof}
Suppose for a contradiction that $Z$ is not totally bounded.
Then there is $\epsilon>0$ such that $Z$ cannot be covered by finitely many $4\epsilon$-balls.
We construct inductively an infinite sequence $z_{1},z_{2},\dots$ of points in $Z$ such that the $2\epsilon$-balls $B_{2\epsilon}(z_{i})$ are pairwise disjoint: 
Fix $z_1\in Z$ arbitrary. Suppose that $z_{i}$ are chosen for $i=1,\dots,k$. By choice of $\epsilon$, there is $z_{{k+1}}\not\in \bigcup_{i=1}^kB_{4\epsilon}(z_{i})$. Hence $B_{2\epsilon}(z_{{k+1}})$ is disjoint from all $B_{2\epsilon}(z_{i})$ for $i=1,\dots,k$.
As $G$ acts transitively on $Z$, there are $g_1, g_2,\dots\in G$ such
that $g_iB_{\epsilon}(z_1)\subseteq B_{2\epsilon}(z_{i})$. 
Hence $\mu(B_{\epsilon}(z_1))=\mu(g_iB_{\epsilon}(z_1))\leqslant\mu(B_{2\epsilon}(z_{i}))$, and this tends to $0$ as $i\to\infty$, because the 
$B_{2\epsilon}(z_{i})$ are pairwise disjoint.
This argument applies to each ball $B_\epsilon(z_1)$, $z_1\in Z$, and as $Z$ is separable, this would imply that $\mu$ is the zero-measure.
\end{proof}
Now we can finish the discussion from Remark~\ref{remark:discussion} with the following lemma:
\begin{lemma}\label{lemma:2.5}
If $\lambda(\Xnull)=1$, then
$\Seq_{0}\subseteq\bigcap_{N\in\cN}K_{0}(N)$.
\end{lemma}
\begin{proof}
It suffices to show that $\Seq_{0}\subseteq S_{0}^{\pi_{N}}= K_{0}(N)$ for each
${N\in\cN}$. In view of Remark~\ref{remark:discussion} this follows from  Lemma~\ref{lemma:elementary-2}, applied to the non-trivial $G$-invariant Borel probability measure $\mu:=\lambda\circ\pi_{N}^{-1}$ on $X_{N}$.
\end{proof}

\begin{lemma}\label{lemma:2.6}
Assume that $\lambda$ has full topological support in $X$. Then
$\bigcap_{N\in\cN}K_{0}(N)\subseteq \Qnull$.
\end{lemma}
\begin{proof}
We follow the arguments in the proof of \cite[Theorem 8]{Auslander1988} for the minimal case: 
Let $(x,y)\in\bigcap_{N\in\cN}K_{0}(N)\subseteq \Xnull^2$,
let $V\in\cOsq$ be a neighbourhood of $\Deltanull$,
and let $A\in\cO$ be a neighbourhood of $x$. Without loss we can assume that $A\times A\subseteq V$, because $(x,x)\in\Deltanull$. Define
\begin{displaymath}
N:=\overline{\bigcup_{g\in G}g\left(\{y\}\times A\right)}^{\,\cOsq}.
\end{displaymath}
Obviously, $N$ is $\cOsq$-closed and $G$-invariant, i.e. $N\in\cN$, and
$\{y\}\times A\subseteq N$
(consider $g=e$), so that $A\subseteq N_y$.
As $(x,y)\in K_{0}(N)$, it follows that $\lambda(A\setminus N_x)\leqslant\lambda(N_{y}\setminus N_x)\leqslant d_N(x,y)=0$. As $A$ is open and $N_x$ is closed, this implies 
$A\subseteq N_x$ and hence $(x,x)\in\{x\}\times A\subseteq N$. Therefore, there are $g\in G$ and $x'\in A$ such that $(gy,gx')\in A\times A$. Hence $(gx',gy)\in A\times A\subseteq V$.
As $\Xnull$ is dense in $X$ by assumption and as $A\in\cO$, one can choose $x'\in A\cap \Xnull$. As $y\in \Xnull$, this proves that $(x,y)\in \Qnull$.
 (Observe that this proves a bit more, namely that in (\ref{eq:PR-characterization}) of Lemma~\ref{lemma:PR-characterization} the point $y'$ can be chosen to be equal to $y$. Interchanging the roles of $x$ and $y$, one could instead choose $x'=x$, see \cite[Ch.~9, Cor.~9]{Auslander1988}.)
\end{proof}

\begin{theorem}\label{theo:all-equal}
Suppose there exists a $G$-invariant Borel probability measure $\lambda$ on $X$ 
with full topological support and with $\lambda(\Xnull)=1$.
Then 
\begin{equation*}
\Qnull=S^\ast_{0}=\Seq_{0}=\bigcap_{N\in\cN}K_{0}(N)=\bigcap_{N\in\cN}K_{t}(N)\cap\Xnull^2\ ,
\end{equation*}
where the equivalence relations $K_{0}(N)$ are determined as above using $\lambda$. 
In particular, $\Seq_{0}=\Seq_t\cap\Xnull^2$.
(Observe that these identities hold for each such measure $\lambda$.)
\end{theorem}
\begin{proof}
We have $\Qnull\subseteq S^\ast_{0}$ by definition, $S^\ast_{0}\subseteq\Seq_{0}$ by Lemma~\ref{lemma:2.2}, $\Seq_{0}\subseteq\bigcap_{N\in\cN}K_{0}(N)=\bigcap_{N\in\cN}K_{t}(N)\cap\Xnull^2$ by Lemma~\ref{lemma:2.5}, and 
 $\bigcap_{N\in\cN}K_{0}(N)\subseteq \Qnull$ by Lemma~\ref{lemma:2.6}.
\end{proof}

\subsection{Existence of maximal equicontinuous generic factors}
\label{subsec:existence-MEGF}
The natural question that arises now is whether there actually exists a maximal equicontinuous factor of $(X_t,G)$, or,  in the terminology of \cite{HuangYe2012},
a maximal equicontinuous generic factor of $(X,G)$. We precede the proof of this fact with a more technical lemma whose proof uses the axiom of choice.

\begin{lemma}
For each measure $\lambda$ as in Theorem~\ref{theo:all-equal}, there exists an at most countable family $\cN_c\subseteq\cN$ such that
$\bigcap_{N\in\cN}K_{0}(N)=\bigcap_{N\in\cN_c}K_{0}(N)$ for all invariant $\Xnull\subseteq X_t$ with $\lambda(\Xnull)=1$.
\end{lemma}
\begin{proof}
As $X$ is second countable, there is a countable base $O_1,O_2,\dots$ for the topology of $X^2$. Hence, for each $N\in\cN$, there is an index set $J_N\subseteq\N$ such that
$X_t^2\setminus K_{t}(N)=\bigcup_{j\in J_N}O_j\cap X_t^2$. Let $J=\bigcup_{N\in\cN}J_N$. Using the axiom of choice, we can associate with each $j\in J$ a set $N_j\in\cN$ such that $j\in J_{N_j}$, i.e. such that $O_j\cap X_t^2\subseteq X_t^2\setminus K_t(N_j)$. Let $\cN_c=\{N_j:\ j\in J\}$. Then
\begin{displaymath}
\begin{split}
X_t^2\setminus\left(\bigcap_{N\in\cN_c}K_{t}(N)\right)
&\subseteq
X_t^2\setminus\left(\bigcap_{N\in\cN}K_{t}(N)\right)
=
\bigcup_{N\in\cN}X_t^2\setminus K_{t}(N)
=
\bigcup_{j\in J}O_j\cap X_t^2\\
&\subseteq
\bigcup_{j\in J}X_t^2\setminus K_{t}(N_j)
=
\bigcup_{N\in\cN_c}X_t^2\setminus K_{t}(N)
=
X_t^2\setminus\left(\bigcap_{N\in\cN_c}K_{t}(N)\right).
\end{split}
\end{displaymath}
Hence we have equalities everywhere in this chain of inclusions. As $K_t(N)\subseteq X_t^2$ for all $N\in\cN$, this implies 
$\bigcap_{N\in\cN_c}K_t(N)=\bigcap_{N\in\cN}K_t(N)$ and hence
\begin{equation*}
\bigcap_{N\in\cN}K_{0}(N)
=
\Xnull^2\cap\bigcap_{N\in\cN}K_{t}(N)
=
\Xnull^2\cap\bigcap_{N\in\cN_c}K_{t}(N)
=
\bigcap_{N\in\cN_c}K_{0}(N)\,.
\end{equation*}
\end{proof}

The following theorem is the announced more detailed version of 
Theorem~\ref{theo:max-factor-prelim}. The additional information is contained in its part~\ref{item:c}).

\begin{theorem}\label{theo:max-factor}
Suppose that $(X,G)$ is ergodic, i.e. there exists an ergodic $G$-invariant Borel probability measure $\lambda$ on $X$ with full topological support.
\begin{compactenum}[a)]
\item $(X,G)$ has a maximal equicontinuous generic factor $\pi:(X,G)\gen(Z,G)$, where $(Z,G)$ is a compact, metrizable, equicontinuous system,
unique up to isomorphism, and one can choose $(Z,G)$ as a minimal group translation. 
\item \label{item:c}
If $\Xnull$ is a $G$-invariant subset of $X_t$ with $\lambda(\Xnull)=1$, if $\pi_Y:\Xnull\to Y$ is another homomorphism to a minimal equicontinuous compact system $(Y,G)$, then there is a factor map $\pi':(Z,G)\to (Y,G)$ such that $\pi_Y=\pi'\circ\pi|_{\Xnull}$. In particular, $\pi_Y$~extends continuously to~$X_t$. 

\end{compactenum}
\end{theorem}
\begin{proof} Note first that $\lambda(X_t)=1$ by Remark~\ref{remark:first}. \quad\\
a)\ Enumerate the countable set $\cN_c$ from the previous lemma as $\cN_c=\{N_n:n\in\N\}$. Define $D:X_t\times X_t\to\R$ as $D(x,y)=\sum_{n\in\N}2^{-n}d_{N_n}(x,y)$, where $d_{N_n}$ is the $G$-invariant continuous pseudo-metric on $X_t$ associated with $N_n$. Then also $D$ is a $G$-invariant continuous pseudo-metric on $X_t$ and 
\begin{equation}\label{eq:cXtD}
\cX_t^D:=\left\{(x,y)\in X_t^2: D(x,y)=0\right\}
=
\bigcap_{n\in\N}K_t(N_n)
=
\bigcap_{N\in\cN}K_t(N)
\end{equation}
is an $\cOsq_t$-closed $G$-invariant equivalence relation on $X_t$. Let $Z^*:=X_t/\cX_t^D$ and define 
$d_Z([x],[y]):=D(x,y)$ for $x,y\in X_t$. Then $(Z^*,d_Z)$ is a metric space, 
the canonical projection $\pi:X_t\to Z^*$ is continuous,
and just as in Remark~\ref{remark:discussion} one shows that $(Z^*,d_Z)$ is compact and that $G$ acts in a canonical way isometrically on $(Z^*,d_Z)$.
%
As $\cX_t^D=\Seq_t$ by Theorem~\ref{theo:all-equal}, and as
$d_Z([x],[y])=D(x,y)$ for $x,y\in X_t$, the system $(Z,G)$ is a MEGF for $(X,G)$.

We turn to the proof of the uniqueness (up to conjugacy) of the MEGF:
Suppose that $\pi_Y:(X,G)\gen(Y,G)$ is a further MEGF of $(X,G)$.
Then there are
factor maps $\pi':Z\to Y$ and $\pi'':Y\to Z$ such that $\pi_Y=\pi'\circ\pi$ and $\pi=\pi''\circ\pi_Y$. Hence $\pi=\pi''\circ\pi'\circ\pi$, which implies that $(\pi''\circ\pi')|_{\pi(X_t)}=\id_{\pi(X_t)}$.  As $\pi(X_t)$ is dense in $Z$ and as $\pi''\circ\pi':Z\to Z$ is continuous, this implies that $\pi''\circ\pi'=\id_Z$. In particular, $\pi'$ is \oneone, i.e. the factor map $\pi':Z\to Y$ is a homeomorphism.
\\[1mm]
b)\ We have to prove that $(Z,G)$ is maximal among all compact metrizable equicontinuous generic factors of $(X,G)$ (in the strengthened sense of part \ref{item:c}) of the theorem). So suppose there are a  $G$-invariant subset $\Xnull\subseteq X_t$
with $\lambda(\Xnull)=1$ and a homomorphism $\pi_Y:\Xnull\to Y$ to a compact metrizable equicontinuous minimal system $(Y,G)$.
Then  $\cX_t^D\cap\Xnull^2=\bigcap_{N\in\cN}K_{0}(N)=\Seq_{0}\subseteq S^{\pi_Y}_{0}$ by Theorem~\ref{theo:all-equal}. It follows that $\pi_Y$ factorizes over $\pi|_{\Xnull}:\Xnull\to Z$.
\end{proof} 
%
%
\subsection{Weak mixing and maximal equicontinuous generic factors}
\label{subsec:weak-mixing}
\begin{proof}[Proof of Theorem~\ref{theo:weak-mixing}]
a) We follow the proof of Theorem 2.7 in \cite{HuangYe2012}.
Denote by $\pi:X_t\to Z$ the MEGF of $(X,G)$.
As $(X,G)$ is weakly scattering, there is $(x_0,z_0)\in X_t\times Z$ such that $\overline{G(x_0,z_0)}=X\times Z$.
Define $\phi:X_t\times Z\to Z$,
$\phi(x,z)=z^{-1}\pi(x)$. Then $\phi$ is continuous, and $\phi(gx_0,gz_0)=(gz_0)^{-1}\pi(gx_0)=z_0^{-1}g^{-1}g\pi(x_0)=\phi(x_0,z_0)$ for all $g\in G$. Hence $\phi(x,z)=\phi(x_0,z_0)$ for all $(x,z)\in X_t\times Z$. In particular,
$\pi(x)=\phi(x,e)=\phi(x_0,z_0)$ for all $x\in X_t$, so that $Z=\overline{\pi(X_t)}=\{\phi(x_0,z_0)\}$ is a singleton.\\
%
b) Let $(X,G)$ be an ergodic topological dynamical system, and
denote by $\lambda$ any $G$-invariant Borel probability measure on $X$ with $\lambda(X_t)=1$.
Assume that the MEGF of $(X,G)$ is trivial.
That means that the equicontinuous structure relation $\Seq_t$ defined in (\ref{eq:Steq-def}) is maximal, i.e. $\Seq_t=X_t^2$. Hence $K_t(N)=X_t^2$ for all $N\in\cN$ by Lemma~\ref{lemma:2.5}, where, as before, $\cN$ denotes the family of all closed $G$-invariant subsets of $X^2$. Therefore $\lambda(N_{x'}\triangle N_x)=0$ for all $N\in\cN$ and all $x,x'\in X_t$.

In order to prove that $(X,G)$ is weakly mixing we must show that each $N\in\cN$ is either nowhere dense in $X^2$ or equal to $X^2$. So assume that $N\in\cN$ is not nowhere dense, i.e. that $\inn(N)\neq\emptyset$. Then there are open sets $U,V\subseteq X$ such that $U\times V\subseteq N$. Fix any $x_0\in X_t\cap U$. Then $V\subseteq N_{x_0}$, and
$\lambda(N_x\triangle N_{x_0})=0$ for all $x\in X_t$. As $\lambda$ has also full topological support, this implies $V\subseteq N_x$ for all $x\in X_t$, i.e. $X_t\times V\subseteq N$. 
Let $W:=\bigcup_{g\in G}gV$. Then $W$ is open and $G$-invariant, and $W$ is dense in $X$, because $(X,G)$ is topologically transitive. It follows that
\begin{equation*}
X^2=\overline{X_t\times W}
=\overline{\bigcup_{g\in G} g(X_t\times V)}\subseteq\overline{\bigcup_{g\in G}gN}=\overline{N}=N\ .
\end{equation*}
\end{proof}

\section{Maximal equicontinuous generic factors and weak model sets}
\label{sec:MEF-for-WMS}

The dynamics of weak model sets are an excellent testing ground for the relevance of MEGFs. We start by summarizing some essential notations and results from \cite{KR2015}.

\subsection{Some recollections on weak model sets}\label{assnot}
\paragraph{Assumptions and notations}
\begin{enumerate}[(1)]
\item $G$ and $H$ are \emph{locally compact second countable abelian groups} with Haar measures $m_G$ and $m_H$. Then the product group $G\times H$ is locally compact second countable abelian as well, and we choose $m_{G\times H}=m_G\times m_H$ as Haar measure on $G\times H$. 
\item $\LL\subseteq G\times H$ is a \emph{cocompact lattice}, i.e., a discrete subgroup whose quotient space $(G\times H)/\LL$ is compact. Thus $\hX:=(G\times H)/\LL$ is a compact second countable abelian group.
Denote by $\piG:G\times H\to G$ and $\piH:G\times H\to H$ 
the canonical projections. We assume that 
$\piG|_\LL$ is \oneone and that
$\piH(\LL)$ is dense in $H$.
\item $G$ acts on $G\times H$ by translation: $gx:=(g,0)+x$.
\item 
Elements of $G\times H$ are denoted as $x=(x_G,x_H)$, elements of $\hX$ as 
$\hx$ or as $x+\LL=(x_G,x_H)+\LL$, when a representative $x$ of $\hx$ is to be stressed. We normalise the Haar measure $m_\hX$ on $\hX$ such that $m_\hX(\hX)=1$. Thus $m_\hX$ is a probability measure. 

\item The \emph{window} $W$ is a compact subset of $H$. We assume that $m_H(W)>0$.
\end{enumerate}

\paragraph{Consequences of the assumptions}\label{en:ass}

\begin{enumerate}[(1)]
\item Being locally compact second countable abelian groups, $G$, $H$ and $G\times H$  are  metrizable with a translation invariant metric with respect to which they are complete metric spaces. In particular they have the Baire property.
As such groups are $\sigma$-compact, $m_G$, $m_H$ and $m_{G\times H}$ are $\sigma$-finite.
\item As $G\times H$ is $\sigma$-compact, the lattice $\LL\subseteq G\times H$ is at most countable.  Note that $G\times H$ can be partitioned by shifted copies of the relatively compact fundamental domain $X$. This means that $\LL$ has a positive finite point density $\dL=1/m_{G\times H}(X)$. We thus have $m_\hX(\hat A)=\dL\cdot m_{G\times H}(X\cap (\pihX)^{-1}(\hat A))$ for any measurable $\hat A\subseteq \hX$, where $\pihX: G\times H\to \hX$ denotes the quotient map. As a factor map between topological groups, $\pihX$ is open.
\item $\LL$ acts on $(H,m_H)$ by $h\mapsto\ell_H+ h$ metrically transitively, i.e., for every measurable $A\subseteq H$ such that $m_H(A)>0$ there exist at most countably many $\ell^i\in \LL$ such that $m_H((\bigcup_i (\ell_H^i+A))^c)=0$,  see \cite[Ch.~16, Ex.~1]{Kharazishvili2009}.
\item The action $\hx\mapsto(g,0)+\hx$ of $G$ on $\hX$ is minimal and uniquely ergodic.

\item Denote by $\cM$ and $\cMG$ the spaces of all locally finite measures on $G\times H$ and $G$, respectively. They are endowed with the topology of vague convergence. 
As $G$ and $G\times H$ are complete metric spaces, this is a Polish topology, see \cite[Thm.~A.2.3]{Kallenberg2001}. 
\end{enumerate}

\paragraph{The objects of interest}
The pair $(\LL,W)$ assigns to each point $\hx\in \hX$ a discrete point set in $G\times H$. 
Such point sets $P$ are identified with the measures $\sum_{y\in P}\delta_y\,\in\cM$.
More precisely:
\begin{enumerate}[(1)]
\item For $\hx=x+\LL\in\hX$ define
\begin{equation}\label{eq:nuW-def}
\nuW(\hx):=\sum_{y\in (x+\LL)\cap(G\times W)}\delta_y\ .
\end{equation}
It is important to understand $\nuW$ as a map from $\hX$ to $\cM$. 
The canonical projection $\piG:G\times H\to G$ projects measures $\nu\in\cM$ to measures $\piG_*\nu$ on $G$ defined by $\piG_*\nu(A):=\nu((\piG)^{-1}(A))$. We abbreviate 
\begin{equation}\label{eq:nuWG-def}
\nuWG:=\piG_*\circ\nuW:\hX\to\cMG
\end{equation}
The set of continuity points of $\nuW$ and $\nuWG$ is a dense $G_\delta$-subset of $\hX$.
\item\label{item:spaces}
 Denote by 
\begin{itemize}
\item[-] $\MW$ the vague closure of $\nuW(\hX)$ in $\cM$,
\item[-] $\MWG$ the vague closure of $\nuWG(\hX)$ in $\cMG$,
\end{itemize}
The group $G$ acts continuously by translations on all these spaces:
$(g\nu)(A):=\nu(g^{-1}A)$.
As $\nuW(\hx)(g^{-1}A)=(g\nuW(\hx))(A)=\nuW(g\hx)(A)$, it is obvious that all $\nuW(\hx)$ are uniformly translation bounded, and it follows from \cite[Thm.~2]{BaakeLenz2004} that both spaces are compact.
\item $\QM:=m_\hX\circ {\nuW}^{-1}$ and 
$\QMG:=m_\hX\circ (\nuWG)^{-1}$ are the \emph{Mirsky measures} on 
$\MW$ and $\MWG$, respectively. Note that $\QMG=\QM\circ (\piG_*)^{-1}$.
\end{enumerate}
Observe that $\MWG$ is the space that contains the weak model sets of primary interest.

\subsection{The MEGF of the Mirsky measure}
The following facts are taken from \cite{KR2015} and \cite{KR2018}:
\begin{enumerate}[(1)]
\item Facts about the MEF \cite[Theorem 1]{KR2015}:
\begin{compactenum}[a)]
\item If $\inn(W)\neq\emptyset$, then $(\hX,G)$ is the MEF of  $(\MW,G)$.
\item If $\inn(W)=\emptyset$, then the MEFs of $(\MW,G)$ and $(\MWG,G)$ are trivial.
\end{compactenum}
\item Facts about the KF \cite[Theorems 2]{KR2015} and \cite[Theorem B1]{KR2018}:
\begin{compactenum}[a)]
\item $(\hX,m_\hX, G)$ is the KF of $(\MW,\QM,G)$. Even more, both systems are isomorphic.
\item If the window $W$ is Haar aperiodic, then the same is true for the system $(\MWG,\QMG,G)$.
\end{compactenum}
Here $W$ is \emph{Haar aperiodic}, if $m_H\left((h+W)\triangle W\right)=0$ implies $h=0$.
\end{enumerate}
\quad\\
Denote by $X\subseteq\MW$ the topological support of $\QM$ and by $\XG\subseteq\MWG$ that of $\QMG$. Although both sets may be strictly contained in their ambient spaces, they capture the most important aspects of the dynamics. 
\begin{corollary}\label{coro:KR2015}
Statements (1) and (2) above remain true for the subsystems $(X,G)$ and $(\XG,G)$.
\end{corollary}
\begin{proof}
For (2) this is trivial, because $\QM(X)=\QMG(\XG)=1$. For (1) observe first, that $(\MW,G)$ has a unique minimal subsystem \cite[Lemma 6.3]{KR2015}, so this system is contained in $X$. In case a)
it  is an almost automorphic extension of $(\hX,G)$ \cite[Theorem 1a]{KR2015}, so that $(\hX,G)$ is its MEF. But then $(\hX,G)$ is also the MEF of $(X,G)$. In case b), the minimal system is a fixed point, so that the MEF of any subsystem containing this fixed point is trivial.
\end{proof}

The facts listed as (1) above, and also Corollary~\ref{coro:KR2015}, provide no 
useful information on the MEF of the systems $(\MWG,G)$ and $(\XG,G)$, respectively. This changes completely when the
MEGF of $(\XG,G)$ is considered. (Recall that $\XG$ denotes the support of the Mirsky measure on $\MWG$.)

We denote by $W_{reg}$ the topological support of $m_H|_W$. This is the smallest closed subset of $W$ which has full Haar measure inside $W$.
We say that the window $W$ is \emph{Haar regular}, if $W_{reg}=W$ \cite[Def.~3.10]{KR2018}.
A Haar regular window is Haar aperiodic if and only it is \emph{aperiodic}, i.e. if $W+h=W$ implies $h=0$ \cite[Sec.~3.2]{KR2018}.

\begin{theorem}\label{theo:cut-project-MEGF}
Suppose that $W$ is Haar regular.
\begin{compactenum}[a)]
\item $(\hX,G)$ is the MEGF of $(X,G)$.
\item If the window $W$ is aperiodic, then the same is true for the system $(\XG,G)$.
\end{compactenum}
\end{theorem}

\begin{remark}
If the window is not aperiodic, the MEGF of $(\XG,G)$ is the $G$-action on factor group of~$\hX$: the periods of $W$ must be factored out as in the corresponding theorems from \cite{KR2018}.
\end{remark}

We precede the proof of the theorem with some observations and notations from \cite{KR2015,KR2018}:
The system $(\MWG,G)$ can be extended to the \emph{graph system} $(\GMWG,G)$, where $\GMWG\subseteq\hX\times\MG$ is the closure of the graph of $\nuWG:\hX\to\MG$. This system is a topological joining of $(\hX,G)$ and $(\MWG,G)$.
The natural projections $\pihX_*:\GMWG\to\hX$ and $\tpiG_*:\GMWG\to\MWG$ are continuous. 
Finally, there is a homeomorphism $\Phi:\GMWG\setminus(\hX\times\{\0\})\to\MW\setminus\{\0\}$ that commutes with the respective actions of $G$. The natural projection $\piG_*:\MW\to\MWG$ satisfies $\tpiG_*=\piG_*\circ\Phi$.
\footnote{See \cite[Lem.~5.3]{KR2015} and the discussion around \cite[Lem.~2.1]{KR2015}.
In the notation from that paper, $\Phi=\piGH_*\circ(\pihXG_*)^{-1}$, and $\tpiG_*$ is also denoted by $\piG_*$.} The composition $\hat\pi:=\pihX_*\circ\Phi^{-1}:\MW\setminus\{\0\}\to\hX$ is the (continuous) map defined in \cite[Def.~5.5]{KR2015}. In view of \cite[Lemma~4.4]{KR2015},
it satisfies 
\begin{equation}\label{eq:domination}
\nu\leqslant\nuW(\hat\pi\nu)\quad\text{for all}\quad\nu\in\MW\setminus\{\0\},
\end{equation}

Next we recall one more concept from \cite[Sec.~3.2 and 4]{KR2018}: 
Denote $\widetilde{\MW}:=\{\nu\in\cM: \nu\leqslant\nuW(\hx)\text{ for some }\hx\in\hX\}$.
Then $\MW=\overline{\nuW(\hX)}\subseteq\widetilde{\MW} $, because $\nuW$ is upper semi-continuous, and the projection $\piG_*$ extends naturally from $\MW$ to $\widetilde{\MW}$.
For each $\nu\in\widetilde{\MW}$, $\piH_*\nu$ is a measure
%
on $H$, and we denote the topological support of this measure by $\SH(\nu)$. We thus have
\begin{equation}\label{eq:SH(nu)}
\SH(\nu)=\supp(\piH_*(\nu))=\overline{\piH(\supp(\nu))}\subseteq W\ .
\end{equation}
It is advantageous to view $\SH$ as a map from $\widetilde{\MW}\setminus\{\0\}$ to $\KW$, the space of all non-empty compact subsets of~$W$, which is equipped with the topology generated by the Hausdorff distance.

Denote by $\XG_t$ the set of topologically transitive points of $\XG$. 
\begin{lemma}\label{lemma_X_t-0}
If $\piG_*\nu\in \XG_t$ for some $\nu\in\widetilde{\MW}$, then $W_{reg}\subseteq\SH(\nu)$.
\end{lemma}
\begin{proof}
Suppose first that $m_H(W')=m_H(W)$. Then $m_H(W'\cap W_{reg})=m_H(W)$
so that $W_{reg}\subseteq W_{reg}\cap W'\subseteq\SH(\nu)$.

Suppose now that
$m_H(W)-m_H(W')>0$. We will derive a contradiction:
Let $\nu\in\widetilde{\MW}\setminus\{\0\}$ and recall from \eqref{eq:domination} that $\nu\leqslant\nuW(\hat\pi\nu)$.
Then
\begin{equation*}
\nu\{x\}=1
\quad\Rightarrow\quad 
\nuW(\hat\pi\nu)\{x\}=1\text{ and } x_H\in W'
\quad\Rightarrow\quad
\nuWprime(\hat\pi\nu)\{x\}=1
\quad\text{for each }x\in G\times H
\end{equation*}
just by the definition of $W'$, and so $\nu\leqslant\nuWprime(\hat\pi\nu)$. 
Hence, given any tempered van Hove sequence $(A_n)_{n\in\N}$ of subsets of $G$,
the upper density of any $S_g\nu$ w.r.t. $(A_n)_n$ is bounded 
in the following way\cite[Thm.~3]{KR2015} \footnote{This is based on a previous result by Moody \cite{Moody2002}.}: $\forall\epsilon>0\ \exists n_0\in\N\ \forall n\geqslant n_0\ \forall g\in G :$
\begin{equation*} 
\frac{(S_g\nu)(A_n\times H)}{m_G(A_n)}
\leqslant
\frac{(S_g\nuWprime(\hat\pi\nu))(A_n\times H)}{m_G(A_n)}
=
\frac{\nuWprime(T_g(\hat\pi\nu))(A_n\times H)}{m_G(A_n)}
\leqslant
\dL\cdot m_H(W')+\epsilon.
\end{equation*}
This applies in particular to
$\epsilon:=\frac{\dL}{2}\left(m_H(W)-m_H(W')\right)>0$ and yields
\begin{equation*}
\frac{(S_g(\piG_*\nu))(A_n)}{m_G(A_n)}
=
\frac{(\piG_*(S_g\nu))(A_n)}{m_G(A_n)}
=
\frac{(S_g\nu)(A_n\times H)}{m_G(A_n)}
\leqslant 
\dL\cdot m_H(W')+\epsilon
=
\dL\cdot m_H(W)-\epsilon
\end{equation*}
for all $g\in G$. Hence all points in the orbit closure of $\piG_*\nu$ have Banach density at most $\dL\cdot m_H(W)-\epsilon$. As $\XG=\supp(\QMG)$ and as $\QMG$-a.a.~$\nuG\in\XG$ have density equal to $\dL\cdot m_H(W)$, this contradicts the assumption that $\piG_*\nu$ is a topologically transitive point in $\XG$.
\end{proof}

\begin{lemma}\label{lemma:X_t}
Suppose that $W$ is Haar regular and and aperiodic.
\begin{compactenum}[a)]
\item If $\piG_*\nu\in \XG_t$ for some $\nu\in\widetilde{\MW}$, then $(\piG_*)^{-1}\{\piG_*\nu\}=\{\nu\}$.
\item $(\piG_*)^{-1}|_{\XG_t}:\XG_t\to\MW\setminus\{\0\}$ and
$\hpiG:\XG_t\to \hX$, $\hpiG:=\hat\pi\circ(\piG_*)^{-1}|_{\XG_t}$ are well defined maps. Both are continuous (when $\XG_t\subseteq \XG$ is equipped with the subspace topology).
\item 
If $\nuG\in \XG_t$, then $\nuG\leqslant\nuWG(\hpiG\nuG)$.
\item 
If $\nuG\in \XG_t$ and $\nuG\leqslant\nuWG(\hx)$ for some $\hx\in\hX$, then $\hpiG\nuG=\hx$.
\end{compactenum}
\end{lemma}
\begin{proof}
a) This follows from \cite[Lemma~4.6]{KR2018}, because Haar regularity of $W$ and Lemma~\ref{lemma_X_t-0} imply that $\SH(\nu)=W$.
\\[2mm]
b) As each $\nuG\in \XG_t$ is the image of some $\nu\in\MW\setminus\{\0\}$ under $\piG_*$, the well definedness follows from part~a).
For the continuity of $(\piG_*)^{-1}|_{\XG_t}$ it suffices to observe that the preimage of any closed set $A\subseteq \MW$ under this map is closed in $\XG_t$:
\begin{displaymath}
\left\{\nuG\in\XG_t:(\piG_*)^{-1}(\nuG)\in A\right\}
=
\XG_t\cap\piG_*(A),
\end{displaymath}
and as $A$ is a closed subset of the compact metric space $\MW$ and $\piG_*:\MW\to\MWG$ is continuous, the set $\piG_*(A)$ is closed.
\\[2mm]
c) Let $\nuG\in\XG_t$ and denote $\nu:=(\piG_*)^{-1}\nuG$. As $\nu\in\MW\setminus\{\0\}$, we have $\nu\leqslant\nuW(\hat\pi\nu)$. Hence
$
\nuG=\piG_*\nu\leqslant\piG_*(\nuW(\hat\pi\nu))=\nuWG(\hpiG\nuG)
$.
\\[2mm]
d) Let $\hx\in\hX$ and $\nuG\in\XG_t$ with $\nuG\leqslant\nuWG(\hx)=\piG_*(\nuW(\hx))$,
and denote $\nu:=(\piG_*)^{-1}\nuG$. 
Define $\tilde{\nu}\in\MW\setminus\{\0\}$ by $\tilde{\nu}\{(g,h)\}=1$, if $\nuW(\hx)\{(g,h)\}=1$ and $\nuG\{g\}=1$, otherwise $\tilde{\nu}\{(g,h)\}=0$.
Then
$\piG_*\tilde{\nu}=\nuG$ and $\tilde{\nu}\leqslant\nuW(\hx)$. As $\piG_*\nu=\nuG=\piG_*\tilde{\nu}$, 
Lemma~4.4 from \cite{KR2018} guarantees the existence of some $d\in H$ such that 
$\tilde{\nu}(A)=\nu(A-(0,d))$ for each Borel subset $A$ of $G\times H$, and consequently $\SH(\nu)=\SH(\tilde{\nu})+d$. But $\SH(\nu)=\SH(\tilde{\nu})=W$ as shown in the proof of part a), so that $W=W+d$, and the aperiodicity of $W$ implies that $d=0$.
Hence $\tilde{\nu}=\nu$, so that $\nu\leqslant\nuW(\hx)$. On the other hand, $\hat\pi\nu$ is the unique point in $\hX$ for which $\nu\leqslant\nuW(\hat\pi\nu)$ \cite[Lem.~5.4]{KR2015}, so that $\hx=\hat\pi\nu$, where
$\hat\pi\nu=\hat\pi((\piG_*)^{-1}\nuG)=\hpiG\nuG$.
\end{proof}

\begin{proof}[Proof of Theorem~\ref{theo:cut-project-MEGF}]
$(X,G)$ and $(\XG,G)$ are ergodic systems in the sense of Definition~\ref{def:ergodic}, where the respective Mirsky measures play the role of the measure $\lambda$. Hence Theorem~\ref{theo:max-factor} applies, and both systems have compact abelian groups as MEGFs, call them  $(Z,G)$ and $(\ZG,G)$ with factor maps $\piZ$ and $\piZG$, respectively.\\
a)\ 
$\hat\pi=\pihX_*\circ\Phi^{-1}:(X_t,G)\to(\hX,G)$ is a generic factor map. 
 By Theorem~\ref{theo:max-factor}\ref{item:c}, there is a factor map $\pi':Z\to\hX$ such that 
$\hat\pi=\pi'\circ\piZ$ on $X_t$. 
Let $\hx_0=\pi'(0)$. The  $\pi'=\alpha+\hx_0$ for some group homomorphism $\alpha:Z\to\hX$.\footnote{As $G$ acts on $Z$ and on $\hX$ by translation, there are group monomorphisms $\eta:G\to Z$ and $\hat\eta:G\to\hX$ such that $gz=z+\eta(g)$ and $g\hx=\hx+\hat\eta(g)$ for all $z\in Z, \hx\in\hX$ and $g\in G$. 
Consider $\alpha:=\pi'-\hx_0$. 
For $g,g'\in G$ we have
\begin{equation*}
\begin{split}
\alpha(\eta(g)+\eta(g'))
&=
\pi'(\eta(g+g'))-\hx_0
=
\pi'(0)+\hat{\eta}(g+g')-\hx_0
=
\hat{\eta}(g)+\hat{\eta}(g').
\end{split}
\end{equation*}
In other words: $\alpha|_{\eta(G)}:\eta(G)\to\hat{\eta}(G)$ is a group homomorphism, and as $\alpha:Z\to\hX$ is continuous and $\eta(G)$ is dense in $Z$, this shows that
$\alpha:Z\to\hX$ is a group homomorphism.} This shows that, for each $\hx\in\hX$, the set $(\pi')^{-1}\{\hx\}$ is a coset of $\ker(\alpha)$, and in order to prove that $\pi':(Z,G)\to(\hX,G)$ is a homeomorphism it suffices to observe that $\card((\pi')^{-1}\{\hx\})=1$ for each continuity point $\hx$ of $\nuW$: for such a point, 
$\{\nuW(\hx)\}=\hat\pi^{-1}\{\hx\}=(\piZ)^{-1}((\pi')^{-1}\{\hx\})$.
Hence $\hat\pi:(X_t,G)\to(\hX,G)$ is a maximal generic factor. 
\\[1mm]
b)\ 
$\hpiG=\hat\pi\circ(\piG_*)^{-1}:\XG_t\to\hX$ is a continuous map by Lemma~\ref{lemma:X_t}b. As $\hat\pi$ and $\piG_*$ commute with the respective actions of $G$, the map $\hpiG:(\XG_t,G)\to(\hX,G)$ is indeed an equicontinuous generic factor.


Next, $\piZG\circ\piG_*:(X_t,G)\to(\ZG,G)$
is an equicontinuous generic factor.
As $\hat\pi:(X_t,G)\to(\hX,G)$ is the MEGF of $(X,G)$ by part a) of the present theorem,
Theorem~\ref{theo:max-factor}\ref{item:c} guarantees the existence of a factor map $\pi':(\hX,G)\to(\ZG,G)$ such that $\piZG\circ\piG_*=\pi'\circ\hat\pi$
on $X_t$. On the other hand, Theorem~\ref{theo:max-factor}\ref{item:c} applies as well to the equicontinuous generic factor
$\hat\pi\circ(\piG_*)^{-1}:(\XG_t,G)\to(\hX,G)$, i.e. there is a factor map $\pi'':(\ZG,G)\to(\hX,G)$ such that $\hat\pi\circ(\piG_*)^{-1}=\pi''\circ\piZG$ on
$\XG_t$. Hence $\hat\pi=\pi''\circ\piZG\circ\piG_*=\pi''\circ\pi'\circ\hat\pi$ on $X_t$, so that
$\pi''\circ\pi'=\id_{\hX}$ on the dense subset $\hat\pi(X_t)$ of $\hX$.
As $\pi''\circ\pi'$ is continuous, this shows that $\pi''\circ\pi'=\id_{\hX}$,
in particular $\pi'$ is \oneone.
Being a factor map from $\ZG$ onto $\hX$, $\pi'$ thus is a homeomorphism,
so that $(\hX,G)$ is (isomorphic to) the MEGF $(\ZG,G)$ of $(\XG,G)$.
\end{proof}

\section{Applications to $\cB$-free dynamics}\label{sec:B-free}

$\cB$-free dynamical systems form a (very) special sub class of dynamical systems generated by weak model sets. We discuss the MEGF of these systems 
and use it to prove that the centralizer of $\cB$-free systems of Erd\"os type is trivial.

\subsection{A recollection of facts on $\cB$-free systems}

For any given set $\cB\subseteq\N$ one can define its \emph{set of multiples} $\cM_\cB=\bigcup_{b\in\cB}b\Z$ and the set of \emph{$\cB$-free numbers} $\cF_\cB=\Z\setminus\cM_\cB$. Investigations into the structure of $\cM_\cB$ or, equivalently, of $\cF_\cB$ have a long history (see \cite{BKKL2015} for references), and dynamical systems theory provides some useful tools for such studies. One way to see this is to interpret such systems as  weak model sets, where $G=\Z$,
$H$ is a closed subgroup of $\tilde H:=\prod_{b\in\cB}\Z/b\Z$, namely the closure of the canonically embedded integers: $H=\overline{\Delta(\Z)}$ where $\Delta(n)=(n,n,n,\dots)\in\tilde{H}$. 
Finally, $\cL=\{(n,\Delta(n)):N\in\Z\}$ and $W=\{h\in H: h_b\neq0\ \forall b\in\cB\}$. It is easy to see that in this case $\hX$ is isomorphic to $H$, 
so that instead of $\nuWG:\hX\to\cMG$ one simply looks at $\nuWG:H\to\cMG$, where
$\nuWG(h)=\sum_{n\in\Z}\delta_{n}\cdot1_W(h+\Delta(n))$ $(h\in H)$. Then $\nuWG(0)=\sum_{n\in\Z}\delta_{n}\cdot1_W(\Delta(n))=\sum_{n\in\cF_\cB}\delta_n$.

In the literature on $\cB$-free dynamics (e.g.~\cite{Kulaga-Przymus2014,BKKL2015,KKL2017}) these point measures on $\Z$ are represented as 0-1-sequences. This means that the map $\nuWG$ is replaced by a map
$\varphi:H\to\{0,1\}^\Z$ defined by $(\varphi(h))_n=1_W(T^nh)=1_W(h+\Delta(n))$, where $\Delta:\Z\to H$ is the natural embedding of $\Z$ into $H$ and, correspondingly, the set $\MWG$ is replaced by the set
$X_\varphi:=\overline{\varphi(H)}$. The systems $(\MWG,\Z)$ and $(X_\varphi,\Z)$, with the respective left shifts as $\Z$-actions are obviously isomorphic dynamical systems. The Mirsky measure $\QMG$ is denoted by $\nu_\eta$, because it is (quasi-)generic for the point $\eta:=\varphi(\Delta(0))\in X_\varphi$.

If the set $\cB\subseteq\Z$ is \emph{taut} (a basic regularity property whose definition is recalled in the next subsection), then the window is always \emph{Haar aperiodic} \cite{KKL2017}.
Hence $(H,\Z)$ is the MEGF of $(X_\varphi,\Z)$ by Theorem~\ref{theo:cut-project-MEGF}, where $\Z$ acts on $H$ by $h\mapsto h+\Delta(n)$,
and $(H,m_H,\Z)$ is also the Kronecker factor of $(X_\varphi,\nu_\eta,\Z)$, and there are many other invariant measures with the same KF. That this need not be the case for all ergodic invariant 
probability measures on $(X_\varphi,\Z)$ is demonstrated in the following example. 

\begin{example}\label{example:B-free-KF-MEGF}
Consider an \emph{Erd\"os set} $\cB$ as studied in \cite{Kulaga-Przymus2014}, where the elements of $\cB$ are pairwise co-prime and $\sum_{b\in\cB}1/b<\infty$.
Even for this rather special class of systems (which includes the square-free numbers)
one can find ergodic invariant measures with full topological support, for which the KF is not supported by the MEGF. The following construction of such examples uses a result from \cite{Kulaga-Przymus2014} whose proof relies on \cite[Theorem 2]{Furstenberg1995}.

 Let $\kappa$ be an ergodic shift-invariant probability measure on $\{0,1\}^\Z$, and denote by 
$\nu_\varphi*\kappa:=M_*(\nu_\varphi\otimes\kappa)$ the ``convolution'' of $\nu_\varphi$ and $\kappa$, where $M:(\{0,1\}^\Z)^2\to\{0,1\}^\Z$ is the coordinate-wise multiplication \cite[Section 2]{Kulaga-Przymus2014}. Corollary~3.15 together with Remark~3.16 of that reference shows that the system
$(X_\varphi,\nu_\eta*\kappa,\Z)$ is isomorphic to the direct product of the systems system $(X_\varphi,\nu_\eta,\Z)$ and $(\{0,1\}^\Z,\kappa,\Z)$, whenever the latter system is (isomorphic to) an irrational rotation. In order to make sure that $\nu_\varphi*\kappa$ has the same topological support as $\nu_\varphi$ itself, it suffices to produce a 0-1-coding of an irrational rotation for which each block of 1's (of arbitrary length) has positive probability.

To this end fix any irrational number $\alpha\in(0,1]$ and a sequence $(J_n)_{n>0}$ of open subintervals of $\R/\Z$ with length $|J_n|=(2(2n+1)2^n)^{-1}$. Define
\begin{equation*}
E:=\bigcup_{n=1}^\infty\bigcup_{k=-n}^{n}(J_n+k\alpha)\ .
\end{equation*}
Denote the Lebesgue measure on $\R/\Z$ by $\lambda$.
Then $0<\lambda(E)\leqslant1/2$. The coding map $\varphi_E:\R/\Z\to\{0,1\}^\Z$, $x\mapsto(1_E(x+k\alpha))_{k\in\Z}$, is $\lambda$-almost surely \oneone\ \footnote{Since I could not locate this statement in the literature, I provide a sketch of a proof: It suffices to prove that $\bigcup_{k\in\Z}(R_\alpha\times R_\alpha)^k(E\times E^c)\subseteq(\R/\Z)^2$ has full 1-dimensional Lebesgue measure on the line $L_\delta:=\{(x,x+\delta): x\in\R/\Z\}$ for every irrational $\delta$. By ergodicity of $R_\alpha\times R_\alpha$ along $L_\delta$, it suffices to show that the measure is not zero. Therefore suppose for a contradiction that $(E\times E^c)\cap L_\delta$ has Lebesgue measure $0$. Then $E\subseteq E+\delta$ up to measure $0$, which contradicts $0<\lambda(E)<1$ in view of the irrationality of $\delta$.}, and the probability, that this coding produces only 1's at positions $-n,\dots,n$ is at least $|J_n|>0$.

Formally this construction can be written as a kind of model set, where the internal space $H$ is replaced by $H\times(\R/\Z)$ and where the set $W\times E$ is taken as a window. Observe however, that $E$ is an open dense subset of $\R/\Z$, so that the closure of this window would be $W\times(\R/\Z)$, which by itself is a window that reproduces precisely the original system. Hence this construction is far from any weak model set.
\end{example}

\begin{remark}
Example~2 in \cite[Section 2.2.2]{Kulaga-Przymus2014} shows that many $\cB$-free systems support invariant ergodic measures $P$ with the following two properties:
\begin{compactenum}[i)]
\item The KF of $P$ is bigger than the KF of the Mirsky measure and is hence not supported by the MEGF of the system. 
\item $P$ is obtained as the Mirsky measure of a compact sub-window of the original window.
\end{compactenum}
These measures do not have full topological support, however. 
\end{remark}

\subsection{Generalized Erd\"os-type $\cB$-free systems have trivial centralizers}
\label{subsec:Centralizers}
Recall that $X_\varphi=\overline{\varphi(H)}$,
$\eta=\varphi(\Delta(0))$, and denote $X_\eta:=\overline{\varphi(\Delta(\Z))}=\overline{\{S^n\eta:n\in\Z\}}$, where $S:X_\varphi\to X_\varphi$ denotes the left shift. In the sequel we denote the natural $\Z$ action on $H$ by $T:H\to H$, $h\mapsto h+\Delta(1)$. We noted already that
\begin{enumerate}[(P1)]
\item\label{item:P1} 
$S\circ\varphi=\varphi\circ T$, but in general $\varphi:H\to X_\varphi$ is not continuous.
\end{enumerate}

For the rest of this section we assume that the set \emph{$\cB$ is taut} (see e.g. \cite{Hall1996,BKKL2015}). This means that for each $b\in\cB$ the logarithmic asymptotic density $\ddelta(\cM_{\cB\setminus\{b\}})$ is strictly smaller than $\ddelta(\cM_\cB)$, where $\ddelta(\cM_\cB):=\lim_{n\to\infty}\frac{1}{\log n}\sum_{k\leqslant n,k\in\cM_\cB}k^{-1}$
denotes the logarithmic density of this set, which is known to exist by the Theorem of Davenport and Erd\"os \cite{DE1936,DE1951}.

\paragraph{General assumption:}
$\cB$ is taut and contains an infinite pairwise co-prime subset. \\[2mm]
Under this assumption $(X_\eta,S)$ is proximal \cite{BKKL2015} and
\begin{enumerate}[({A}1)]
\item\label{item:A1}  $X_\eta=X_\varphi=\supp(\nu_\eta)$, where $\nu_\eta=m_H\circ\varphi^{-1}$ denotes the Mirsky measure. We call this set simply $X$ and denote by $X_t$ the topologically transitive points in $X$. As $m_H$, and hence also $\nu_\eta$, is ergodic, we have $m_H(\varphi^{-1}(X_t))=1$.
\item\label{item:A2}  $X$ is hereditary, i.e. if $x\in X$ and $y\in\{0,1\}^\Z$ are such that $y\leqslant x$, then also $y\in X$.
\item\label{item:A3}  The window $W$ is Haar regular and aperiodic.
\end{enumerate}
Properties (\rm A\ref{item:A1}) and (\rm A\ref{item:A2}) are finally proved in \cite{Keller-tautness}, but the proof relies significantly on previous work in \cite{BKKL2015} and to some extent also on \cite[Prop.~2.2]{KKL2017}. Property (\rm A\ref{item:A3}) is the combination of \cite[Thm.~A and Prop.~5.1]{KKL2017}.

As $(X,S)$ is proximal with unique fixed point $0^\Z$ \cite{BKKL2015}, the MEF of this system is trivial and useless for the centralizer problem.
Here we will use the MEGF of $(X,S)$ instead
to give a purely dynamical alternative proof of a result of Mentzen~\cite{Mentzen2017} (who assumes that $\cB$ is of Erd\"os type).\footnote{There are also some unpublished  notes by Lema\'n{}czyk et al. providing a related proof but also using explicitly arithmetic properties.}  See also Remark~\ref{remark:mentzen} below.  
\begin{theorem}[Trivial centralizer, see also~\cite{Mentzen2017}]\label{theo:trivial-centralizer}
Suppose that $\cB$ is taut and contains an infinite pairwise co-prime subset.
If $F:X\to X$ is  a homeomorphism that commutes with $S$, then $F\circ S^k=\id_X$ for some $k\in\Z$.
\end{theorem}

\subsubsection{The role of the MEGF}
Some arguments below are based on our Lemma~\ref{lemma:X_t}. For the convenience of the reader we rewrite the two relevant statements of that lemma using the special notation for $\cB$-free systems:
\begin{lemma}\label{lemma:X_t-B}
Assume that $\cB$ is taut and denote by $X_t$ the set of topologically transitive points of $X$. There is a well defined map $\pi:X_t\to H$ with the following properties:
\begin{compactenum}[a)]
\item $\pi:X_t\to H$ is continuous.
\item If $x\in X_t$, then $x\leqslant\varphi(\pi x)$.
\item If $x\in X_t$ and $x\leqslant\varphi(h)$, then $h=\pi x$.
\end{compactenum}
\end{lemma}

\begin{corollary}\label{coro:B-free-MEGF}
$\pi:(X_t,S)\to(H,T)$ is the MEGF of $(X,S)$. In particular, 
\begin{enumerate}[(P1)]
\setcounter{enumi}{1}
\item\label{item:P2}  $S(X_t)=X_t$ and $\pi\circ S_{|X_t}=T\circ\pi$.
\end{enumerate}
\end{corollary}
\begin{proof}
If $x\in X_t$, then also $Sx\in X_t$, and as $x\leqslant\varphi(\pi x)$ by {Lemma~\ref{lemma:X_t-B}b}, also $Sx\leqslant S(\varphi(\pi x))\overset{(\rm P\ref{item:P1})}=\varphi(T(\pi x))$, so that $\pi(Sx)=T(\pi x)$ by {Lemma~\ref{lemma:X_t-B}c}. The continuity of $\pi:X_t\to H$ is proved in {Lemma~\ref{lemma:X_t-B}a}. Hence $\pi$ is an equicontinuous generic factor map. In view of (\rm A\ref{item:A3}) and Theorem~\ref{theo:cut-project-MEGF}b it is the MEGF.
\end{proof}

Let 
$F:X\to X$ be a \emph{continuous and surjective map that commutes with $S$}. Then
\begin{enumerate}[(P1)]
\setcounter{enumi}{2}
\item\label{item:P3}  $F(X_t)= X_t$, 
\end{enumerate}
so that the map $\pi\circ F|_{X_t}:X_t\to H$ is a generic factor map. Hence, by Corollary~\ref{coro:B-free-MEGF} and Theorem~\ref{theo:max-factor}b, there is a (continuous!) factor map $f:(H,T)\to(H,T)$ satisfying
\begin{enumerate}[(P1)]
\setcounter{enumi}{3}
\item\label{item:P4}  $\pi\circ F|_{X_t}=f\circ\pi$.
\end{enumerate}  
Combining (\rm P\ref{item:P2}) and (\rm P\ref{item:P4}) we get
$T(f(\pi x))=T(\pi(Fx))=\pi(S(Fx))=\pi(F(Sx))=f(\pi(Sx))=f(T(\pi x))$ for all $x\in X_t$, and as both $f$ and $T$ are continuous\footnote{Indeed, they are both translations.} and $\pi(X_t)$ is dense in $H$, we have
\begin{enumerate}[(P1)]
\setcounter{enumi}{4}
\item\label{item:P5}   $T\circ f=f\circ T$.
\end{enumerate}

%
\begin{remark}\label{remark:mentzen}
The whole setting described so far applies to general weak model sets, except assumptions (A\ref{item:A1}) -- (A\ref{item:A3}) which, in the context of $\cB$-free dynamics, follow from $\cB$ being taut and containing an infinite pairwise co-prime subset. But, of course, they are formulated in dynamical terms without any recourse to arithmetics. This distinguishes our approach from the one by Mentzen \cite{Mentzen2017}.
\end{remark}

\begin{proposition}\label{prop:main-2}
If $F:X\to X$ is a continuous and surjective map that commutes with $S$, then
there is $k\in\Z$ such that $F(S^kx)\leqslant x$ for all $x\in X$.
\end{proposition}
\begin{proof}
As in the first lines of the proof of Lemma~2 in \cite{Mentzen2017}
it follows that $F(0^\Z)=0^\Z$ and
$F(0^\infty10^\infty)\neq0^\Z$ (one just uses the heredity and proximality of $X$.) 
Denote by $\rho:\{0,1\}^{[-a:a]}$ the block map that defines $F$. Then $\rho(0^{[-a:a]})=0$, because $F(0^\Z)=0^\Z$, and
as $F(0^\infty10^\infty)\neq0^\Z$, there is $k\in\Z$ such that
$F(S^k(0^\infty10^\infty))[0]=S^k(F(0^\infty10^\infty))[0]=F(0^\infty10^\infty)[k]=1$. 

From now on we consider $F\circ S^k$ instead of $F$ and denote this new map again by $F$. The new $F$ also satisfies (\rm P\ref{item:P3}) -- (\rm P\ref{item:P5}), when the map $f$ is replaced by $f\circ S^k$. 
We denote the block map for the new $F$ by $\rho$ again so that, defining $u:=0^a10^a$, we have $\rho(u)=1$.
We will prove that $F(x)\leqslant x$ for all $x\in X$ for the new $F$.

Our first goal is to show that
\begin{equation}\label{eq:goal-1}
\varphi(f(h))[0]\geqslant\varphi(h)[0]\quad\text{for all }h\in H_0:=\varphi^{-1}(X_t).
\end{equation}
Let $h\in H_0$, so that $x:=\varphi(h)\in X_t$.
If $\varphi(h)[0]=0$, there is nothing to prove. So we may assume that $\varphi(h)[0]=1$. 
Define $y\in\{0,1\}^\Z$ by
$y[-a:a]=u$ and $y[n]=x[n]$ for $|n|>a$. Then $y[n]\leqslant x[n]=\varphi(h)[n]$ for all $n\neq0$
and $y[0]=1=\varphi(h)[0]$, so that $y\leqslant\varphi(h)$ and $y\in X$ by heredity (\rm A\ref{item:A2}). 
As $y[n]=x[n]$ for all $n\in\Z$ except at most finitely many and as $x\in X_t$, also $y\in X_t$.
Hence {Lemma~\ref{lemma:X_t-B}c} applies to $y$, so that $\pi y=h$.
We see that $F(y)\in X_t$ by (\rm P\ref{item:P3}), so that
{Lemma~\ref{lemma:X_t-B}b} can be applied to to $F(y)$:
\begin{displaymath}
F(y)\leqslant \varphi(\pi(Fy)).
\end{displaymath}
It follows that
\begin{equation*}
\varphi(f(h))[0]
=
\varphi(f(\pi y))[0]
\overset{(\rm P\ref{item:P4})}=
\varphi(\pi(Fy))[0]\geqslant F(y)[0]=\rho(y[-a:a])=\rho(u)=1=\varphi(h)[0],
\end{equation*}
that is \eqref{eq:goal-1}.
Now, for arbitrary $n\in\Z$, 
$\varphi(T^nh)=S^n(\varphi(h))\in S^n(X_t)=X_t$, so that \eqref{eq:goal-1} applies also to the point $T^nh$ as well, whence
\begin{equation*}
\begin{split}
\varphi(f(h))[n]
&\overset\qquad=
S^n\left(\varphi(f(h))\right)[0]
\overset{(\rm P\ref{item:P1})}=
\varphi\left(T^n(f(h))\right)[0]
\overset{(\rm P\ref{item:P5})}=
\varphi\left(f(T^nh)\right)[0]\\
&\overset{\eqref{eq:goal-1}}\geqslant
\varphi\left(T^nh\right)[0]
\overset{(\rm P\ref{item:P1})}=
S^n(\varphi(h))[0]
=
\varphi(h)[n].
\end{split}
\end{equation*}
Hence
\begin{equation}\label{eq:goal-2}
\varphi(f(h))\geqslant\varphi(h)\quad\text{for all }h\in H_0.
\end{equation}
As $m_H(H_0)=1$ in view of (\rm A\ref{item:A1}), this shows that $\varphi\circ f=\varphi$ $m_H$-a.s. Hence we have for $m_H$-almost all $h\in H$
\begin{equation*}
F(\varphi(h))
\overset{\text{Lemma~\ref{lemma:X_t-B}b}}\leqslant
\varphi(\pi(F(\varphi(h))))
\overset{(\rm P\ref{item:P4})}=
\varphi(f(\pi(\varphi(h))))
\overset{\text{Lemma~\ref{lemma:X_t-B}c}}=
\varphi(f(h))
=\varphi(h),
\end{equation*}
equivalently, $F(x)\leqslant x$ for $\nu_\eta$-a.a. $x\in X$. As $\supp(\nu_\eta)=X$ by (\rm A\ref{item:A1}) and as $F$ is continuous, this shows that $F(x)\leqslant x$ for all $x\in X$.
\end{proof}

\begin{proof}[Proof of Theorem~\ref{theo:trivial-centralizer}]\quad\\
If $F:X\to X$ is even a homeomorphism, then Proposition~\ref{prop:main-2} applies to $F^{-1}$ as well. 
Hence
there are $k,\ell\in\Z$ such that $F(S^kx)\leqslant x$ and $F^{-1}(S^{-\ell} x)\leqslant x$ for all $x\in X$. Thus, for all $x\in X$,
\begin{equation}\label{eq:nearly-finished}
x=(F^{-1}\circ S^{-\ell})\circ (F( S^{\ell} x))
\leqslant 
F( S^{\ell} x)
=
F(S^{k}( S^{\ell-k} x))
\leqslant S^{\ell-k}x.
\end{equation}
Applied to the point $x=0^\infty10^\infty\in X$ this shows that $\ell=k$, and we conclude from \eqref{eq:nearly-finished} that
$x=F(S^kx)$ for all $x\in X$.

\end{proof}

\end{document}